\documentclass{amsart}
\usepackage[dvips]{graphics}
\usepackage{amsthm,amssymb,amsmath,enumerate,graphics,epsf}
\title{Strongly maximal matchings in infinite weighted graphs }

\author{Ron Aharoni}
\address{Department of Mathematics\\Technion, Haifa\\ Israel 32000}
\email[Ron Aharoni]{ra@tx.technion.ac.il}

\author{Eli Berger}
\address{Department of Mathematics\\Princeton University and \\Department of Mathematics
\\Technion, Haifa\\ Israel 32000} \email[Eli Berger]{berger@cri.haifa.ac.il}

\author{Agelos Georgakopoulos}
\address{Universit\"at Hamburg}
\email[Agelos Georgakopoulos]{http://www.math.uni-hamburg.de/home/georgakopoulos/}

\author{Philipp Spr\"ussel}
\address{Universit\"at Hamburg}
\email[Philipp Spr\"ussel]{spruessel@math.uni-hamburg}

\begin{document}
\thanks{\noindent The research of the first author was
supported by grant no. 780-04 of the Israel Science Foundation, by
the Technion's research promotion fund, and by the Discont Bank
chair.}

\maketitle

\newcommand{\pf}{{\it Proof.}~}
\newcommand{\hetz}{{^\curvearrowright}}
\newcommand{\quo}{/}
\newcommand{\extends}{{\succcurlyeq}}
\newcommand{\fextends}{{\vec{\extends}}}
\newcommand{\extended}{{\preccurlyeq}}
\newcommand{\fextended}{{\vec{\extended}}}

\newtheorem{claim}{Claim}

\newcommand{\mymargin}[1]{
  \marginpar{%
    \begin{minipage}{\marginparwidth}\small%
      \begin{flushleft}%
        #1%
      \end{flushleft}%
    \end{minipage}%
  }%
}%

\newcommand{\cig}{{\mathcal I}(G)}

\newcommand{\ca}{\mathcal A}
\newcommand{\cb}{\mathcal B}
\newcommand{\cc}{\mathcal C}
\newcommand{\cd}{\mathcal D}
\newcommand{\ce}{\mathcal E}
\newcommand{\cf}{\mathcal F}
\newcommand{\cg}{\mathcal G}
\newcommand{\ci}{\mathcal I}
\newcommand{\cj}{\mathcal J}
\newcommand{\ck}{\mathcal K}
\newcommand{\cl}{\mathcal L}
\newcommand{\cm}{\mathcal M}
\newcommand{\cn}{\mathcal N}
\newcommand{\co}{\mathcal O}
\newcommand{\cp}{\mathcal P}
\newcommand{\cq}{\mathcal Q}
\newcommand{\cs}{\mathcal S}
\newcommand{\cgr}{\mathcal R}
\newcommand{\ct}{\mathcal T}
\newcommand{\cu}{\mathcal U}
\newcommand{\cv}{\mathcal V}
\newcommand{\cx}{\mathcal X}
\newcommand{\cy}{\mathcal Y}
\newcommand{\cz}{\mathcal Z}
\newcommand{\cw}{\mathcal W}
\newcommand{\tlp}{T_{\lambda^+}}
\newcommand{\cyta}{\mathcal Y \langle T_\alpha \rangle}

\newcommand{\uuu}{\bigsqcup}
\newcommand{\giv}[1]{\ensuremath{H_i(#1)}}

\newcommand{\givp}[1]{\ensuremath{H_{i+1}(#1)}}

\newcommand{\enp}{\hfill \Box}
\newcommand{\tn}{\tilde{N}}
\newcommand{\ch}{\ensuremath{\mathcal H}}

\newcommand{\etabar}{\overline{\eta}}
\newcommand{\heta}{\hat{\eta}}
\newcommand{\tg}{\tilde{\gamma}}
\newcommand{\ttau}{\tilde{\tau}}
\newcommand{\rn}{\mathbb{R}^n}

\theoremstyle{plain}
\newtheorem{theorem}{Theorem}[section]
\newtheorem{lemma}[theorem]{Lemma}
\newtheorem{fact}[theorem]{Fact}
\newtheorem{conjecture}[theorem]{Conjecture}
\newtheorem{corollary}[theorem]{Corollary}
\newtheorem{assertion}[theorem]{Assertion}
\newtheorem{proposition}[theorem]{Proposition}
\newtheorem{observation}[theorem]{Observation}
\newtheorem{problem}[theorem]{Problem}

\theoremstyle{remark}
\newtheorem{notation}[theorem]{Notation}
\newtheorem{remark}[theorem]{Remark}
\newtheorem{definition}[theorem]{Definition}
\newtheorem{convention}[theorem]{Convention}
\newtheorem{assumption}[theorem]{Assumption}

\theoremstyle{remark}
\newtheorem{example}[theorem]{Example}

\def\sm{\setminus}

\newcommand{\sss}{****} \newcommand{\fig}[1]{Figure~\ref{#1}}
\newcommand{\Fig}{Figure}

\newenvironment{txteq}{\begin{equation}\begin{minipage}[c]{0.8\textwidth}\em}
{\end{minipage}\ignorespacesafterend\end{equation}\ignorespacesafterend}

\newcommand{\comment}[1]{}

\newcommand{\showFig}[2]{ \begin{figure}[htbp] \centering \noindent
\epsfbox{#1.eps} \caption{\small #2} \label{#1} \end{figure} }

\newcommand{\N}{\mathbb N} \newcommand{\R}{\mathbb R}

\newcommand{\Tr}[1]{Theorem~\ref{#1}}
\newcommand{\ProR}[1]{Proposition~\ref{#1}}
\newcommand{\Sr}[1]{Section~\ref{#1}} \newcommand{\g}{\ensuremath{G\ }}
\newcommand{\Lr}[1]{Lemma~\ref{#1}} \newcommand{\defi}[1]{\emph {#1}}

\newcommand{\oo}{\ensuremath{\omega}}
\newcommand{\OO}{\ensuremath{\Omega}}
\newcommand{\G}{\ensuremath{G}}
\newcommand{\RRR}{\ensuremath{\mathcal R}}
\newcommand{\TTT}{\ensuremath{\mathcal T}}

\providecommand{\MAX}{\text{MAX}}

\renewcommand{\theenumi}{(\roman{enumi})}
\renewcommand{\labelenumi}{\theenumi}

\begin{abstract}
Given an assignment of weights $w$ to the edges of a graph $G$, a matching $M$ in $G$ is called {\em strongly $w$-maximal} if for any matching $N$ there holds
$\sum\{w(e) \mid e \in N \setminus M\} \le \sum\{w(e) \mid e \in
M \setminus N\}$. We prove that if $w$ assumes only finitely many values all of which are rational then $G$ has a strongly $w$-maximal matching.
\end{abstract}

\section{introduction}
Infinite min-max theorems are rather weak when stated in terms of
cardinalities. Cardinalities are too crude a measure to
capture the duality relationship. To exemplify this point,
consider Menger's theorem, the first combinatorial theorem that
was cast in the form of a min-max equality. Formulated in terms of
cardinalities, it states that given two sets, $A$ and $B$ in an
infinite graph, the maximal cardinality $\kappa$ of a family of
disjoint $A$--$B$ paths is equal to the minimal cardinality of a
vertex-set separating $A$ from $B$. This is easy to prove: if $\kappa$ is finite then it follows from the finite version of the theorem, and if it is infinite then we can take a maximal set $\mathcal P$ of disjoint $A$--$B$ paths, and choose the set of vertices appearing in $\mathcal P$ as our separating set. A more succinct formulation, capturing the duality in its
full strength is the following, which is known as the Erd\H{o}s-Menger Conjecture:

\begin{theorem}[\cite{infinitemenger}]\label{infmenger}
Given two vertex-sets, $A$ and $B$ in an infinite graph, there exists a
set $F$ of disjoint $A$--$B$ paths and an $A$--$B$ separating set
$S$ such that $S$ consists of a choice of precisely one vertex
from every path in $F$.
\end{theorem}

This formulation is tantamount to requiring the complementary
slackness conditions to hold between the two dual objects.

A similar situation occurs when studying matchings in infinite graphs. It is easy to prove the existence of a maximal matching with respect to cardinality, however, it is possible to find matchings that are maximal in a stronger sense:

\begin{definition}
A matching $M$ in a hypergraph $H$ is said to be {\em strongly
maximal} if $|N \setminus M| \le |M \setminus N|$ for any
matching $N$.
\end{definition}

The notion of strong maximality is closely related to duality
results. Namely, it is used to prove duality results, and
conversely, a main tool in proofs of existence of strongly maximal
matchings is duality theorems. In particular, Theorem
\ref{infmenger} is equivalent (in the sense of easy derivation, in
both directions) to the statement that in the hypergraph of
$A$--$B$ paths (a path being identified with its vertex set) there
exists a strongly maximal matching. The set $S$ in Theorem~\ref{infmenger} is a strongly {\em minimal cover} in this hypergraph,
where the notion of strong minimality is defined in an analogous
way. It is interesting to note that not every strongly minimal
separating set $S$ has a corresponding matching $F$ as in the
theorem. An example showing this is the bipartite graph $G$ with
sides $A$ and $B$, where $A=\{a_0,a_1,a_2, \ldots, \}$, \quad $B=
\{b_1,b_2, \ldots\}$, and $E(G)=\{(a_i,b_i) \mid 1\le i <\omega\}
\cup \{(a_0,b_i) \mid 1 \le i < \omega\}$. The side $A$ is a
strongly minimal separating set, but there is no $F$ corresponding
to it as in the theorem, since, easily, $A$ is unmatchable.

The main result of \cite{infinitetutte} implies:

\begin{theorem}\label{strongmaximailityingraphs} In any graph there exists a strongly maximal
matching. \end{theorem}

As expected, the theorem follows from a duality result. The proof
will be given in Section~\ref{sec:smmatchings}. Beyond graphs very
little is known. The main conjectures on the notions of strong
maximality and strong minimality are the following:

\begin{conjecture}\label{smaxandminconj}
  In any hypergraph with finitely bounded size of edges there exists a strongly maximal matching and a strongly minimal cover of the vertex set by edges of the hypergraph.
\end{conjecture}

\begin{conjecture}
  \label{sminbyind}
  In every graph there exists a strongly minimal cover of the vertex set by independent sets.
\end{conjecture}

An interesting conjecture that would follow from a positive answer to Conjecture~\ref{sminbyind} is the following:

\begin{conjecture} \label{fishscale} In any poset of bounded width
there exists a chain $C$ and a partition of the vertex set into
independent sets, all meeting $C$.
\end{conjecture}

In this paper we are going to extend Theorem
\ref{strongmaximailityingraphs} to graphs with weighted edges. Here and throughout the paper, for a set $F$ of edges we define $w[F]:=\sum_{e \in F}w(e)$. Let $G$ be a graph and $w: E(G) \to \R$ an assignment of weights to the edges of $G$ fixed throughout this section.

\begin{definition} A
matching $M$ in $G$ is called {\em strongly $w$-maximal} if
$w[N\setminus M] \le w[M \setminus N]$ for any matching $N$ in $G$ with
$|M \setminus N|,|N \setminus M|<\infty$.
\end{definition}

\begin{theorem}\label{main} If $w$ assumes only finitely many
values all of which are rational, then $G$ has a strongly $w$-maximal
matching.
\end{theorem}

On the way to the proof of Theorem \ref{main} we shall prove:

\begin{theorem}\label{sminpm}
Suppose that $G$ is complete and $w$ assumes only
finitely many values all of which are rational. Then there exists a strongly $w$-minimal perfect matching, or a strongly $w$-minimal almost perfect
matching.
\end{theorem}

A strongly $w$-minimal perfect or almost perfect matching $M$ is a perfect or almost perfect matching that is strongly $w$-minimal (which is defined analogously to strongly $w$-maximal) among all perfect and almost perfect matchings in $G$ (i.e.\ there is no perfect or almost perfect matching $N$ with $|M \setminus N|,|N \setminus M|<\infty$ and $w[N\setminus M] < w[M \setminus N]$). Note that such a matching will, in general, not be strongly $w$-minimal among all matchings in $G$.

As we shall see, \Tr{sminpm} is best possible in the sense that
it false if we allow irrational weights or if we demand the matching to be perfect rather than almost perfect.

\section{Definitions}

We will be using the terminology of \cite{diestelBook05}.

The {\em support} of a matching $M$, denoted by $supp(M)$, is the
set of vertices incident with $M$.

Let $M$ be a matching. A path or a cycle $P$ is said to be
$M$-{\em alternating} if one of any two adjacent edges on $P$
lies in $M$.
An $M$-alternating path $Q$ is said to be
{\em finitely improving} (or {\em finitely $M$-improving}) if it
is finite and both its endpoints do not belong to $supp(M)$. It
is said to be {\em infinitely improving} (or {\em infinitely
$M$-improving}) if it is infinite, has one endpoint, and this
endpoint does not belong to $supp(M)$. It is said to be {\em
$M$-indifferent} if it is either two way infinite or it is finite
and has one endpoint in $supp(M)$ and one endpoint outside
$supp(M)$.

Given two matchings $M$ and $N$, a path or cycle is said
to be $M$--$N$-{\em alternating} if it is both $M$-alternating and
$N$-alternating. For example, an $M$--$N$-alternating path may consist
of only one edge belonging to both $M$ and $N$.

Given to sets $K$, $L$ of edges, their \emph{symmetric difference} is the set $K \triangle L := (K \cup L) \setminus (K\cap L)$.

A graph $C$ is called {\em almost matchable} if $C-v$ has a
perfect matching for some $v \in V(C)$. It is called {\em
uniformly almost matchable} if $C-v$ has a perfect matching for
{\em every} $v \in V(C)$.

For a graph $G$ and a set of vertices $U$ of $G$ we write $G[U]$ for the subgraph of $G$ induced by the vertices in $U$.

\section{Strongly maximal matchings in graphs}
\label{sec:smmatchings}

In this section we prove \Tr{strongmaximailityingraphs} and develop some tools for the proof of \Tr{main}.

\begin{lemma}\label{strongly_max_iff_no_finiteap}
A matching $M$ is strongly maximal if and only if there does not
exist a finitely improving $M$-alternating path.
\end{lemma}

\begin{proof}
If $P$ is a finitely improving $M$-alternating path then the
matching $M \triangle E(P)$ witnesses the fact that $M$ is not
strongly maximal. For the converse, assume that $M$ is not
strongly maximal, namely there exists a matching $N$ such that $|N
\setminus M| > |M \setminus N|$. It is easy to see that $M \triangle N$ spans a set $\cf$ of $M$--$N$ alternating paths and cycles. Now $N \setminus M=\bigcup_{Q\in
\cf}(N\cap E(Q) \setminus M \cap E(Q))$ and $M \setminus
N=\bigcup_{Q\in \cf}(M\cap E(Q) \setminus N \cap E(Q))$, thus the
inequality $|N \setminus M| > |M \setminus N|$ implies the
existence of a path $Q$ in $\cf$ such that $|N\cap E(Q)|>|M \cap
E(Q)|$. Then, $Q$ is a finitely improving $M$-alternating path.
\end{proof}

We will use the following
result from \cite{aharoniziv}, stating that the classical
Gallai-Edmonds decomposition theorem is valid also for infinite
graphs. A graph $C$ is called {\em factor
critical} if it is uniformly almost matchable but does not
have a perfect matching.

\begin{theorem}\label{gallaiedmonds}
In any graph $G$ there exists a set of vertices
$T$, a set $\cf$ of factor critical components of $G-T$, and
an injective function $F: T \to \cf$ such that
\begin{enumerate}
\item for every $t
\in T$ there exists a vertex $v(t)$ of $F(t)$ connected to
$t$ in $G$, and
\item $G-T-\bigcup_{F \in \cf} V(F)$ has a perfect
matching.
\end{enumerate}
\end{theorem}

\begin{proof}[Proof of Theorem \ref{strongmaximailityingraphs}]
Let $T$ and $\cf$ be as in Theorem \ref{gallaiedmonds}. Let $\cg$ consist of
those elements of $\cf$ belonging to the range of $F$, and let $\ch
= \cf \setminus \cg$. For every $t$ in $T$ let $J_t$ be a perfect
matching of the graph $F(t) - v(t)$. For every $F \in \ch$ choose
an almost perfect matching $J_F$. Let $N$ be a perfect matching in
the graph $G-T-\bigcup_{F \in \cf} V(F)$. We claim that the
matching $M$ defined as $\{t v_t \mid t \in T\} \cup \bigcup_{t \in T} J_t \cup \bigcup_{F \in \ch}J_F \cup N$ is strongly maximal. Suppose not; then, by \Lr{strongly_max_iff_no_finiteap}, there exists a finite improving
$M$-alternating path $Q$. By the construction of $M$ the endpoints of $Q$ are unmatched
vertices $v_1,v_2$ of some $F_1, F_2 \in \ch$ respectively where $F_1 \neq F_2$. Now go along $Q$,
starting at $v_1$. Since $F_1$ is a component of $G-T$, the path
$Q$ can leave $F_1$ only through $T$. Let $t_1$ be the first
vertex of $Q$ in $T$. Since the edge of $Q$ leading
to $t_1$ does not belong to $M$, the edge $e$ of $Q$ leaving $t_1$ does
belong to $M$; let $e=:t_1 u_1$, where $u_1 \in F(t_1)$. But when
$Q$ leaves $F(t_1)$, it is again through an edge not belonging to $M$ that contains a vertex $t_2$ of $T$. Thus, again, the edge of $Q$ leaving $t_2$
belongs to $M$, and continuing this way we see that $Q$ cannot leave $T \cup \bigcup \cg$, contradicting
the fact that $v_2 \in F_2 \in \ch$.
\end{proof}

An even stronger notion than strong maximality of a matching in a
graph is that of {\em having (inclusion-wise) maximal support}. Similarly to the proof of Lemma~\ref{strongly_max_iff_no_finiteap} it is possible to show:

\begin{lemma}\label{maxsupport}
A matching $M$ has maximal support if and only if there does not
exist any (finitely or infinitely) improving $M$-alternating path.
\end{lemma}

In \cite{steffens} the
following stronger version of Theorem~\ref{strongmaximailityingraphs} was proved for countable graphs:

\begin{theorem}
In every countable graph there exists a matching with maximal support.
\end{theorem}

In our proof of \Tr{sminpm} we are going to need the following corollary of \Tr{strongmaximailityingraphs}:

\begin{lemma}\label{stronglymaxcontainingsupport}
For any graph $G$, and every matching $M$ in $G$ there exists a strongly
maximal matching $N$ such that $supp(N) \supseteq supp(M)$.

\end{lemma}

\begin{proof}
Let $K$ be a strongly maximal matching of $G$, which exists by \Tr{strongmaximailityingraphs}. Then, the symmetric difference $K \triangle M$ spans a set $\cg$ of disjoint $M$--$K$-alternating paths and cycles. Let $\cg'\subseteq \cg$ be the set of those elements of $\cg$ that are either finite $K$-indifferent paths or infinitely $K$-improving paths. We can derive a new matching $N$ from $K$ by switching between $K$ and $M$ along all paths in $\cg'$; formally, let $N:= K \triangle \bigcup_{P \in \cg'} E(P)$. Clearly, since there are no finitely $K$-improving paths by Lemma~\ref{strongly_max_iff_no_finiteap}, $supp(N) \supseteq supp(M)$. We claim
that $N$ is strongly maximal.

Suppose not. Then, by Lemma~\ref{strongly_max_iff_no_finiteap}, there exists a finitely
improving $N$-alternating path $Q$. We shall use $Q$ in order to construct a matching $L$ such that $|L \setminus K| > |K
\setminus L|$ contradicting the strong maximality of $K$. As an intermediate step, we first construct a further matching $K'$ by removing finitely many edges from $K$ and adding the same amount of new edges. To define $K'$, we start with $K$ and perform the following operations:

\begin{enumerate}
\item \label{finKindifferent}
  For every finite element $P$ of $\cg'$ incident with $Q$, replace $K \cap E(P)$ by $M \cap E(P)$ (the resulting matching thus coincides with $N$ on $E(P)$; note that $P$ has even length as it is a finite $K$-indifferent path).
\item \label{infKimproving}
  For every infinite element $R$ of $\cg'$ (i.e.\ for every infinitely $K$-improving path in $\cg$) incident with $Q$, let $k=k(R)$ be the last edge on $R$ that lies in $K$ and is incident with $Q$. Replace all edges of
$R$ that lie in $K$ and precede $k$ on $R$, including $k$ itself, by the
edges of $M$ lying on $R$ and preceding $k$.
\end{enumerate}

Let $K'$ be the resulting matching. By construction, $K'$ satisfies
$|K'\setminus K| = |K\setminus K'|<\infty$. Moreover, $K' \cap E(Q)=N \cap E(Q)$ holds by construction and thus $Q$ is a
$K'$-alternating path as it is an $N$-alternating path, and in fact it is a finitely $K'$-improving
one: To prove this, we have to show that the endvertices of $Q$ do not lie in $supp(K')$. As $Q$ is finitely $N$-improving, its endvertices do not lie in $supp(N)$. If an endvertex $v$ of $Q$ does not lie in $supp(K)$, it clearly also does not lie in $supp(K')$ (as $supp(K')\subset supp(K)\cup supp(N)$). On the other hand, if $v$ lies in $supp(K)$ and hence in $supp(K)\setminus supp(N)$, then by the construction of $N$ it is the endvertex of a finite $K$-indifferent path in $\cg'$. This path was considered in~\ref{finKindifferent} and hence $v\notin supp(K')$. Therefore the endvertices of $Q$ do not lie in $supp(K')$ and $Q$ is a finitely $K'$-improving path.

Letting $L= K' \triangle E(Q)$ we thus have $|L \setminus
K'|>|K' \setminus L|$, from which it easily follows that
$|L \setminus K|>|K \setminus L|$, contradicting the fact that $K$ is strongly maximal.
\end{proof}

\section{Strongly maximal weighted matchings}

In this section we prove \Tr{sminpm} and \Tr{main}. Before we do so, let us argue that \Tr{sminpm} is in a way best possible. First, we claim that the requirement that $G$ be a complete graph is essential in it. Indeed, if $G$ is any graph that has an almost perfect matching, then it does not necessarily have an almost perfect strongly $w$-minimal matching.
To see this, consider the graph consisting of a set of paths $P_1,P_2, \ldots$ that have precisely their first vertex $w$ in common, such that each $P_i$ comprises $2i$ edges weighted alternatingly with zeros and ones (starting at $w$ with a zero-weight edge). Any almost perfect matching of this graph that matches $w$ by an edge $e$ can be improved by matching $w$ by the first edge of a $P_j$ with a higher index than the $P_i$ containing $e$, and the almost perfect matching that does not match $w$ can be improved by any almost perfect matching. This example can easily be modified to obtain a graph that has a perfect matching but no perfect strongly $w$-minimal one: add a copy $K$ of $K_{\aleph_0}$ to the graph, identifying the final vertex of each $P_i$ with a distinct vertex of $K$ and let all edges of $K$ have weight $0$.

Next, let us see why we cannot improve \Tr{sminpm} by always demanding a strongly $w$-minimal perfect matching rather than an almost perfect one. Let $G$ be a complete graph of any infinite cardinality, pick a vertex $v\in V(G)$, and let $M$ be a perfect matching of $G - v$. Now let $w(e) = 0$ if $e \in M$ and $w(e)=1$ otherwise. Suppose that $N$ is a strongly $w$-minimal perfect matching of $G$, let $e_1=vw$ be the edge of $N$ matching $v$ and let $e_2=w'y$ be the edge of $N$ matching the vertex $w'$ that lies with $w$ in an edge of $M$. But then, $(N \backslash \{e_1,e_2\}) \cup \{vy,ww'\}$ improves $N$, contradicting the fact that it is strongly $w$-minimal. Thus, $G$ has no strongly $w$-minimal perfect matching.

It is easy to construct counterexamples to \Tr{sminpm} and \Tr{main} if $w$ assumes infinitely many values. At the end of this section we will construct a counterexample in the case that $w$ assumes finitely many values that are not all rational.


\begin{proof}[Proof of \Tr{sminpm}]
Without loss of generality we may assume that all weights are positive, since otherwise we can add a large positive constant to all of them. Since $w$ assumes only finitely many values, we may further assume that all weights are integers. All $M$-alternating paths (for some given matching $M$) considered in this section start with an edge that does not lie in $M$.

Our proof is an adaptation of Edmonds' algorithm for finite graphs (\cite{edmonds}, see also \cite{schrijverBook}). This is a ``primal-dual'' optimisation algorithm, where the primal problem is minimising the total weight of a perfect matching and the dual is maximising the sum of a set of ``potentials'' $\pi_i(U)$ assigned to some vertex sets $U$. In the infinite case though, comparing the total weight of a perfect matching with the sum of the potentials does not help, as both values will in general be infinite. However, in order to show that a matching cannot be locally improved, i.e.\ it is strongly minimal, we will only have to compare finitely many edge weights to the sum of finitely many potentials.

The basic idea of Edmonds' algorithm is the following: In the unweighted case, the problem of constructing a maximal matching reduces to the problem of finding a (finitely) improving $M$-alternating path for a given matching $M$. An improving $M$-alternating path, however, is not easy to construct. On the other hand, $M$-alternating walks are easy to construct, but as they may contain cycles they cannot be used to improve $M$ by taking the symmetric difference. However, if an $M$-alternating walk starting in an unmatched vertex runs into a cycle, then this cycle has to be odd and is thus uniformly almost matchable. In Edmonds' algorithm, such odd cycles are contracted (`shrunk') whenever they occur. At the end of the process the cycles are recursively decontracted using the fact that they are uniformly almost matchable to extend the maximal matching of the graph with contracted vertices to a maximal matching of the original graph.

In the weighted case, one wants to find a minimum-weight perfect matching under the assumption that the graph has a perfect matching. The algorithm starts with considering only the edges of smallest weight. Like in the non-weighted case, the algorithm contracts odd cycles that can occur in alternating walks and it improves the current matching by finding improving alternating paths. When all contractions of odd cycles and improvements of the current matching are done, the algorithm considers some of the edges that had not been considered so far. Whether an edge will be considered or not at a given step depends on the potentials $\pi_i$ mentioned earlier. Unlike the non-weighted case, some sets have to be decontracted during the construction, and again whether a set will be decontracted or not depends on the potentials $\pi_i$.

Our adaptation of Edmonds' algorithm has two major differences: Firstly, we will not only contract odd cycles but some larger sets of vertices (possibly infinite). These sets of vertices will be uniformly almost matchable, which will become important when decontracting. Secondly, we will not improve our matchings by finding improving alternating paths as this might take infinitely many steps. Instead, we will in each step extend our current matching to a strongly maximal matching using Lemma~\ref{stronglymaxcontainingsupport}, then perform contractions, and finally add more edges before we proceed to the next step.

Our construction follows a recursive procedure, in each step $i$ of which we will be manipulating several ingredients:

\begin{itemize}
\item a collection $\OO_i$ whose elements are vertex sets, sets of vertex sets, sets of sets of vertex sets and so on, and an assignment of potentials $\pi_i:\OO_i \to \R$. 
\item an auxiliary graph $G_i$ on $V=V(G)$.
 
    \item an auxiliary graph $G'_i$, having as vertices the maximal sets in $\OO_i$.
    \item an auxiliary graph $H_i(U)$ for each set $U \in \OO_i$, having $U$ as its vertex set. 
\item a matching $M_i$ in $G'_i$. 
\end{itemize}
The elements of $\OO_i$ represent the vertex sets contracted so far. For practical reasons we do not want all elements of $\OO_i$ to be vertex sets but also allow sets of vertex sets, sets of sets of vertex sets, and so on.
The graph $G_i$ will consist of all edges considered in step $i$, while the graph $G'_i$ is obtained from $G_i$ by performing the contractions.
The matchings $M_i$ are to be `unfolded' at the end of the process, to form the desired strongly minimal matching in $G$.

For a set $U$ in $\OO_i$ we denote by $\uuu U$ the set of vertices nested in $U$; formally, a vertex $x\in V(G)$ lies in $\uuu U$ if and only if there is a finite sequence of sets $U_1 \in U_2 \in \dotsb \in U_k$ where $U_k=U$ and $x \in U_1$. The collection $\OO_i$ will be \defi{laminar}, that is, for any $U,W \in \OO_i$ either $\uuu U\cap \uuu W = \emptyset$ or $\uuu U\subseteq \uuu W$ or $\uuu W \subseteq \uuu U$ will hold. Moreover, $\OO_i$ will contain $\{v\}$ for every $v \in V$.

The auxiliary graph  $G_i$ is  defined at each step $i$ by $G_i=(V,E_i)$, where $E_i$ is the set of edges of $G$ for which
\begin{equation} \label{eq:sat}
\sum_{\substack{U\in \OO_i\\ e\in \delta(U)}} \pi_i(U)= w(e)
\end{equation}
holds, where $\delta(U)$ is the set of edges that have precisely one endvertex in $\uuu U$.

Let $\OO_i^{\MAX}$ be the set of maximal elements of $\OO_i$ with respect to containment, and note that $\{\uuu U \mid U\in \OO_i^{\MAX}\}$ is a partition of $V(G)$ as $\OO_i$ is laminar and every vertex $v$ is contained in some $\uuu U$, eg.\ in $\uuu \{v\} = \{v\}$. For $U \in \OO_i$ we now define an auxiliary multigraph $\giv{U}$. The vertices of $\giv{U}$ are the elements of $U$, and for every edge $e=xw$ of $G_i$ such that $x\in \uuu X$ and $w \in \uuu W$ where $X,W$ are distinct elements of $U$ we put an $X$-$W$~edge $e'$ in $\giv{U}$. Throughout the paper we shall not formally distinguish the edges $e$ and $e'$. With this abuse of notation, the auxiliary graph $G'_i$ is defined by $G'_i:=\giv{\OO_i^{\MAX}}$, where $\giv{\OO_i^{\MAX}}$ is defined analogously to $\giv{U}$.

At each step $i$ the following conditions will be satisfied:

\begin{gather}
\label{eq:nonnegative}
\pi_i(U) \geq 0 \text{ for every $U\in
\OO_i$ with $\left\vert\uuu U\right\vert\geq 3$,}\\
\label{eq:undersaturated}
\sum_{\substack{U\in \OO_i\\ e\in \delta(U)}} \pi_i(U)
\leq w(e) \text{ for every $e \in E$,}\\
\label{cond:uam}
\giv{U}\text{ is uniformly almost matchable for every $U\in \OO_i$}.
\end{gather}
The procedure stops in case that $M_i$ is perfect or almost perfect. Then, using condition~(\ref{cond:uam}) we will recursively decontract the sets in $\OO_i$ so as to extend $M_i$ to a perfect or almost perfect matching of $G_i$ (and hence of $G$), and use conditions~\eqref{eq:nonnegative} and \eqref{eq:undersaturated} to prove that it is strongly $w$-minimal in $G$.

To start the inductive definition, we set $\OO_0=\{\{v\} \mid v \in V(G)\}$ and $\pi=\pi_0(\{v\})=0$ for every $v$. By its definition,   $G_0$ contains all $0$-weight edges in $G$; the graph $G'_0$ 
is essentially the same, with the subtle difference that its vertices are singleton sets, and not vertices; and 
the graphs $\giv{U}$ are all trivial, namely they have one vertex each, and no edges.
Finally let $M_0$ be a strongly maximal matching in $G'_0$, the existence of 
which is guaranteed by \Tr{strongmaximailityingraphs}.

Now for $i=0,1,\dotsc$ do the following.

If $M_i$ is perfect or almost perfect then stop the iteration (at the end of this proof we will use $M_i$ to construct the required matching of $G$). So, assume that the set $X'_i$ of vertices unmatched by $M_i$ contains more than one vertex. 

In order to enlarge $M_i$ we now would like to add new edges, i.e.\ to change the $\pi$-values so as to let new edges satisfy~\eqref{eq:sat}. As we want to be able to match vertices in $X'_i$, we could try and increase the $\pi$-values on $X'_i$. But then any edge of $G'_i$ at a vertex in $X'_i$ will fail to satisfy~(\ref{eq:undersaturated}) as it already satisfied~\eqref{eq:sat} before and the $\pi$-value of one of its endpoints has been increased while the other remained the same. Hence we have to decrease the $\pi$-values of all neighbours of $X'_i$ in $G'_i$. Now consider an edge in $M_i$ incident with such a neighbour of $X'_i$. As it satisfied~\eqref{eq:sat} before and the $\pi$-value of at least one of its endvertices has been decreased while the other has not been increased, it will not satisfy~\eqref{eq:sat} in the next step. In order to prevent this loss of matching edges, we have to increase the $\pi$-value of every vertex that is matched in $M_i$ to a neighbour of $X'_i$. Continuing this way, we obtain that we want to increase the $\pi$-value on the set $T'_i$ of all vertices of $G'_i$ that are reachable from $X'_i$ by an even $M_i$-alternating path (possibly trivial), while we want to decrease it on the set $S'_i$ of vertices reachable from $X'_i$ by an odd $M_i$-alternating path.

We could proceed like this if $S'_i$ and $T'_i$ were disjoint, but in general this will not be the case. For instance, the vertices on the odd cycles contracted in Edmonds' algorithm have the property that they are reachable from the set of unmatched vertices by alternating paths both of even and odd lengths. To amend this, we will contract each component of $G'_i-(S'_i\setminus T'_i)$ that contains a vertex of $T'_i$, so as to obtain a new graph $G^*_i$. In this graph, we will be able to perform the desired changes of $\pi$-values.

Formally, let
\begin{equation*}
  \cu_i:= \{ V(C) \mid \text{$C$ is a component of $G'_i-(S'_i\setminus T'_i)$ that contains a vertex in $T'_i$} \},
\end{equation*}
put $\cv_i:=\OO_i\cup\cu_i$, and let $G^*_i := \giv{\cv_i^{\MAX}}$ (where $\cv_i^{\MAX}$ is defined analogously to $\OO_i^{\MAX}$). 
Note that $\cv_i$ is laminar since $\OO_i$ is and $\cv_i\setminus\OO_i = \cu_i$ consists of disjoint subsets of $\OO_i^{\MAX}$.

Let $X_i$ be the set of vertices of $G^*_i$ that are not matched by $M^*_i:=M_i\cap E(G^*_i)$ (which, as we shall see soon, will be a matching in $G^*_i$), let $S_i$ be the set of vertices $s$ of $G^*_i$ for which there is an $M^*_i$-alternating $X_i-s$ path of odd length in $G^*_i$, and let $T_i$ be the set of vertices $t$ of $G^*_i$ for which there is a (possibly trivial) $M^*_i$-alternating $X_i-t$ path of even length.
We claim that: 

\begin{proposition} \label{pro:newOO}
The following assertions are true:
\begin{enumerate}
\item \label{enum:uamatchable}
  $\giv{U}=G'_i[U]$ is uniformly almost matchable for every $U\in\cu_i$;
\item \label{enum:edgesfromU}
  $|M_i \cap \delta(U)|=0$ if $U\cap X'_i\not= \emptyset$ and $|M_i \cap \delta(U)|=1$ otherwise for every $U\in\cu_i$, and
\item \label{enum:newST}
  $S_i=S'_i\setminus T'_i$ and $T_i= \cu_i$.
\end{enumerate}
\end{proposition}

Part~\ref{enum:uamatchable} is simply~(\ref{cond:uam}) for the sets in $\cu_i$, while~\ref{enum:edgesfromU} ensures that $M^*_i$ is a matching in $G^*_i$ (which is trivial in the case of finite graphs, when only odd cycles are contracted) and~\ref{enum:newST} will enable us to increase the $\pi$-values on $T_i$ and decrease them on $S_i$ so as to obtain new edges, in particular at the vertices in $X_i$.

Before we proceed with the proof of \ProR{pro:newOO} let us show how we use it to construct $\OO_{i+1}$, $\pi_{i+1}$, and $M_{i+1}$, the main ingredients of the next step of our construction. By \ProR{pro:newOO}\ref{enum:newST} and the definition of $\cu_i$ we have $S_i \cap T_i = \emptyset$, and moreover
\begin{txteq}\label{eq:neighboursofT}
  If $U\in T_i$ and $U'$ is a neighbour of $U$ in $G_i|\cv_i^{\MAX}$, then $U'\in S_i$.
\end{txteq}

Hence we can define $\pi_{i+1}: \cv_i \to \R$ as follows (in fact we want $\OO_{i+1}$ to be the domain of $\pi_{i+1}$ but $\OO_{i+1}$ is going to be a subset of $\cv_i$):
\begin{equation*}
  \pi_{i+1}(U):=
  \begin{cases}
    \frac12& \text{if }U\in T_i=\cu_i,\\
    \pi_i(U)-\frac12& \text{if }U\in S_i,\\
    \pi_i(U)& \text{otherwise.}
  \end{cases}
\end{equation*}

For every set $U\in S_i$ with $|\uuu U|>1$ and $\pi_{i+1}(U)=0$, remove $U$ from $\cv_i$ to obtain $\OO_{i+1}$. This will later guarantee that~\eqref{eq:nonnegative} is satisfied. Since we have now defined $\OO_{i+1}$ and $\pi_{i+1}$, the graphs $G_{i+1}$ and $G'_{i+1}$ are also defined. It remains to define $M_{i+1}$.

For this purpose, we first show that for every $U\in \cv_i$ the graph $\givp{U}$ is uniformly almost matchable. We distinguish two cases. If $U\in \OO_i$, then we have $\givp{U}=\giv{U}$ because $\pi_i(W)=\pi_{i+1}(W)$ holds for every $W\in U$ since $S_i$ and $T_i$ by definition only contain maximal elements of $\cv_i$, so any relevant edge of $G$ is present in $G_i$ if and only if it is present in $G_{i+1}$. Thus $\givp{U}$ is uniformly almost matchable since $\giv{U}$ is (by \eqref{cond:uam}). For the second case, when $U \in \cu_i = \cv_i\setminus\OO_i$, then by \ProR{pro:newOO} $\giv{U}$ is uniformly almost matchable, and again this implies that $\givp{U}$ is uniformly almost matchable as well since $\pi_i(W)=\pi_{i+1}(W)$ holds for every $W\in U$.

Thus we have proved our claim. In particular, since $\OO_{i+1} \subseteq \cv_i$, this implies by induction:
\begin{proposition}
Condition \eqref{cond:uam} is satisfied.
\end{proposition}

By~\ref{enum:edgesfromU} of \ProR{pro:newOO}, $M^*_i$ is a matching in $G^*_i$. Using the fact that for every $U\in\cv_i\setminus\OO_{i+1}$ the graph $\givp{U}$ is uniformly almost matchable, we extend $M^*_i$ to a matching $N_i$ in $G'_{i+1}$ with $U\subseteq supp(N_i)$ for every $U\in\cv_i\setminus\OO_{i+1}$; this is possible since by~\ref{enum:edgesfromU} of \ProR{pro:newOO} there is precisely one vertex of $U$ that is incident with an edge in $M_i$, and this edge is also in $M^*_i$. By Lemma~\ref{stronglymaxcontainingsupport} there is a strongly maximal matching $M_{i+1}$ in $G'_{i+1}$ with $supp(N_i) \subseteq supp(M_{i+1})$.

Finally, before we switch over to the proof of \ProR{pro:newOO}, let us show that the choice of $N_i$ and $M_{i+1}$ imply that

\begin{txteq}\label{unmatched}
  Every vertex $U$ of $G'_{i+1}$ that is not matched by $M_{i+1}$ is a set of vertices of $G'_i$ (i.e.\ $U\notin\OO_i$) and precisely one of the elements of $U$ is unmatched by $M_i$.
\end{txteq}

This will, at the end of the construction, help us to show that the resulting matching is strongly $w$-minimal.

Indeed, consider such a $U$ and note that $U$ is also unmatched by $N_i$ as $supp(N_i)\subseteq supp(M_{i+1})$. Suppose that $U\in\OO_i$. If $U\in\OO_i^{\MAX}$ then $U\notin X'_i$, since otherwise the definition of $\cu_i$ would imply that there is a set $U'\in\cu_i$ that contains $U$; this would in turn imply that $U'\in T_i$ by~\ref{enum:newST} of \ProR{pro:newOO}, and hence $U'\in\OO_{i+1}$ which contradicts the assumption that $U\in V(G'_{i+1})=\OO_{i+1}^{\MAX}$. Thus $U\notin\OO_i^{\MAX}$. Suppose that $U\in\OO_i \setminus \OO_i^{\MAX}$. As $U$ is a vertex of $G'_{i+1}$ there is a set $U'\ni U$ with $U'\in\cv_i\setminus\OO_{i+1}\subset S_i$. Since all elements of $S_i=S'_i\setminus T'_i$ are matched in $M_i$, they are also matched in $M^*_i$. Thus $U'$ is matched in $M^*_i$ and hence all its elements---in particular $U$---are matched in $N_i$, a contradiction. This proves $U\notin\OO_i$, and by the construction of the graphs $G'_i$ we obtain that $U$ is a set of vertices of $G'_i$. To prove~\eqref{unmatched} it remains to show that there is an element of $U$ that is unmatched in $M_i$. But this follows immediately from \ProR{pro:newOO}\ref{enum:edgesfromU}.

\begin{proof}[Proof of \ProR{pro:newOO}]
  We will derive both~\ref{enum:uamatchable} and~\ref{enum:edgesfromU} from another fact. For this, note first that $\cu_i$ is the set of vertex sets of components of $G'_i[T'_i]$, since any vertex adjacent to a vertex of $T'_i$ in $G'_i$ lies, clearly, in $S'_i \cup T'_i$. Now let $U\in\cu_i$ and $u\in U$; then there is an $x\in X'_i$ and a (possibly trivial) $M_i$-alternating $x-u$ path of even length $P$ in $G'_i$. Moreover, for any neighbour $v\in U$ of $u$, we find a $y\in X'_i$ and a (possibly trivial) $M_i$-alternating $y-v$ path of even length $Q$ in $G'_i$. It is easy to see that $P\cup\{uv\}\cup Q$ either contains an $M_i$-alternating $x-y$ path or an $M_i$-alternating $x-v$ path of even length; indeed, if $P$ and $Q$ are disjoint then $P\cup\{uv\}\cup Q$ is itself an $M_i$-alternating $x$--$y$~path, and otherwise, if $q$ is the first vertex on $P$ that lies in $Q$, then either the path $xPqQy$ or the path $xPqQv$ is $M_i$-alternating. But an $M_i$-alternating path between vertices in $X'_i$ is finitely $M_i$-improving, thus, since $M_i$ is strongly maximal, the latter holds. This proves that any vertex $x$ in $X'_i$ that sends an $M_i$-alternating path of even length in $G'_i$ to some vertex of $U$ sends an $M_i$-alternating path of even length in $G'_i$ to \emph{every} vertex of $U$. In particular, $U$ cannot contain more than one element of $X'_i$.

  Let $x,y\in V(G'_i)$. We say that $x$ \defi{dominates} $y$ if there is an $M_i$-alternating $x$--$y$~path of even length. If a set $X\subset V(G'_i)$ contains the vertices of such a path, we say that \defi{$x$ dominates $y$ via $X$}. We claim that

\begin{txteq}\label{dominatingvx}
  For every $U\in\cu_i$ there is a vertex $x_U\in U$ that dominates every $v\in U$ via $U$.
\end{txteq}

For a vertex $x_U$ as in \eqref{dominatingvx} we say that \defi{$x_U$ dominates $U$}. Clearly~\eqref{dominatingvx} implies that every vertex $v$ in $U-x_U$ is matched by $M_i$ to another vertex in $U-x_U$ (namely, to its predecessor in the $M_i$-alternating $x_U$--$v$~path in $G'_i[U]$ of even length), while $x_U$ either lies in $X'_i$ (i.e.\ is unmatched by $M_i$) or is matched by $M_i$ to a vertex outside $U$. In particular, each $U$ can be dominated by at most one vertex. Moreover, \eqref{dominatingvx} implies~\ref{enum:uamatchable} and~\ref{enum:edgesfromU}: Indeed, consider any set $U\in\cu_i$. For every $v\in U$, the symmetric difference of $M_i$ with the $M_i$-alternating $x_U$--$v$~path of even length in $G'_i[U]$ is a matching of $U-v$, which shows~\ref{enum:uamatchable}. Furthermore, as noted above, $|M_i\cap\delta(U)|=0$ if $x_U\in X'_i$ and $|M_i\cap\delta(U)|=1$ otherwise. Since no vertex in $U-x_U$ lies in $X'_i$ this implies~\ref{enum:edgesfromU}.

For the proof of~\eqref{dominatingvx}, we distinguish two cases. The first case is when $U$ contains a vertex of $X'_i$, say $x$. Recall that there is a vertex in $X'_i$ sending an $M_i$-alternating path of even length to every vertex in $U$, and clearly this vertex must be $x$. We claim that $x$ dominates $U$. Indeed, let $U'$ be a maximal subset of $U$ such that $x$ dominates every $u\in U'$ via $U'$, and suppose that $U'\not= U$. As $G'_i[U]$ is connected, there is a vertex $u\in U\setminus U'$ which has a neighbour $v\in U'$. Every vertex $y\in U'-x$ is matched in $M_i$ to a vertex in $U'$, namely to the penultimate vertex on any $M_i$-alternating $x$--$y$~path in $G'_i[U']$ of even length. Therefore no edge in $\delta(U')$ lies in $M_i$; in particular, $vu$ does not lie in $M_i$. Let $P$ be an $M_i$-alternating $x$--$u$~path of even length (possibly using vertices outside $U$) and let $w$ be its last vertex in $U'$. Then, the first edge of $wPu$ does not lie in $M_i$. Now since there is an $x$--$v$~path of even length in $U'$ it is easy to see that all vertices on $wPu$ lie in $T'_i$ and hence in $U$; moreover, for every $y\in wPu$ there is an $M_i$-alternating $x$--$y$~path in $G'_i[U'\cup V(wPu)]$ of even length, thus $x$ dominates $y$ via $U'\cup V(wPu)$, contradicting the maximality of $U'$.

The second case is when $U\cap X'_i=\emptyset$. Again, recall that there is a vertex $x\in X'_i$ that sends an $M_i$-alternating path of even length in $G'_i$ to every vertex of $U$; let $P$ be an $M_i$-alternating $x-U$ path, and note that it has even length since its penultimate vertex cannot lie in $T'_i$. Let $z$ be the last vertex of $P$ and let $e$ be the last edge of $P$ (hence $e\in M_i$). We claim that $z$ dominates every vertex in $U$. Indeed, let $U'\subset U$ be maximal such that $z$ dominates every $v\in U'$ via $U'$. Consider a vertex $u\in U\setminus U'$ which has a neighbour $v\in U'$. Like in the previous case, no edge in $\delta(U')\setminus \{e\}$, in particular $vu$, lies in $M_i$. Let $Q$ be an $M_i$-alternating $x$--$u$~path of even length, let $y$ be its last vertex outside $U$ and let $f$ be the edge on $Q$ after $y$. Since $y\in S'_i\setminus T'_i$, the path $xQy$ has odd length and hence $f\in M_i$. We claim that there is a vertex on $yQu$ that lies in $U'$. If $y$ is the predecessor of $z$ on $P$, then $f=e$ and $z$ is such a vertex. We may thus assume that $y$ is not the predecessor of $z$ on $P$. This implies that $y$ does not lie on $P$, as otherwise $P$ would have to use $f$ and would hence meet $U$ before $z$. If $yQu$ avoids $U'$, then there is an $M_i$-alternating $x$--$y$~path of even length: go from $x$ to $z$ along $P$, then from $z$ to $v$ within $G'_i[U']$, then use the edge $vu$ and finally along $uQy$ to $y$. But $y\notin T'_i$, a contradiction. Hence $yQu$ has a last vertex $w$ in $U'$, and all vertices of $wQu$ lie in $U$. Now like in the previous case it follows that $z$ dominates every vertex in $U'\cup wQu$ via $U'\cup wQu$, contradicting the maximality of $U'$. This proves~\eqref{dominatingvx}, and hence~\ref{enum:uamatchable} and~\ref{enum:edgesfromU} as discussed above.

A consequence of~\ref{enum:edgesfromU} is
\begin{txteq}\label{intersectingU}
  For every $M_i$-alternating path $P$ starting in $X'_i$ and every $U\in\cu_i$, if $P\cap G'_i[U]$ has more than one vertex then it is a subpath of $P$ whose first edge is not in $M_i$ and whose last edge is an edge of $M_i$ or the last edge of $P$.
\end{txteq}
Indeed, let $P$ and $U$ be as in the statement of~\eqref{intersectingU}, and assume that $P$ contains more than one vertex from $U$. For every vertex $u\in U\cap V(P)$ whose predecessor $v$ on $P$ does not lie in $U$ the edge $vu$ lies in $M_i$, as otherwise $Pv$ would have even length, contradicting the fact that $v\in S'_i\setminus T'_i$. By~\ref{enum:edgesfromU} there is no such $u$ if $U$ contains the starting vertex of $P$, and there is at most one such $u$ otherwise. Therefore, $P\cap G'_i[U]$ is a subpath of $P$, and if the endvertex of $P$ does not lie in $U$, then again by~\ref{enum:edgesfromU} the edge of $P$ from $U$ to $V(G'_i)\setminus U$ does not lie in $M_i$, and hence the last edge of $P\cap G'_i[U]$ does lie in $M_i$.

It remains to show~\ref{enum:newST}. Let us first show $S_i\supset S'_i\setminus T'_i$ and $T_i\supset\cu_i$. Let $v\in S'_i\setminus T'_i$ and pick an $M_i$-alternating path $P$ in $G'_i$ of odd length from a vertex $x\in X'_i$ to $v$. Note that $v$ is not contained in any element of $\cu_i$. Let $U_0$ be the element of $\cu_i$ that contains $x$, and note that $U_0\in X_i$ by~\ref{enum:edgesfromU}. Then by~\eqref{intersectingU} contracting the sets in $\cu_i$ turns $P$ into an $M^*_i$-alternating path $P^*$ in $G^*_i$ of odd length starting in $X_i$, hence $v\in S_i$.

Now let $U\in\cu_i$, pick a vertex $u\in U$ and an $M_i$-alternating path $P$ of even length in $G'_i$ from a vertex $x\in X'_i$ to $u$. Again~\eqref{intersectingU} yields that contracting the sets in $\cu_i$ turns $P$ into an $M^*_i$-alternating path $P^*$ of even length in $G^*_i$ starting in $X_i$, whence $U\in T_i$.

To prove $S_i\subset S'_i\setminus T'_i$ and $T_i\subset\cu_i$, let $P^*$ be an $M^*_i$-alternating path in $G^*_i$ from $U_X\in X_i$ to a vertex $U$ of $G^*_i$; we will use $P^*$ to construct an $M_i$-alternating path $P$ in $G'_i$ whose length has the same parity as that of $P^*$. Let $U_0=U_X,U_1,\dotsc,U_n$ be the vertices in $\cu_i$ that lie (in this order) on $P^*$. Note that if $U\in\cu_i$ then $U_n=U$. For $j>0$ let $u_j$ be the vertex on $P^*$ before $U_j$, and for $j<n$ let $w_j$ be the vertex on $P^*$ after $U_j$. Note that each $u_j$ and each $w_j$ are neighbours of $U_j$ (which is a component of $G'_i-(S'_i\setminus T'_i)$) and hence lie in $S'_i\setminus T'_i$. Each edge $u_jU_j$ in $P^*$ corresponds to an edge $u_jv_j^-$ in $E(G'_i)$ with $v_j^-\in U_j$, while each edge $U_jw_j$ corresponds to an edge $v^+_jw_j$ in $E(G'_i)$. For $j=0,1,\dotsc,n$ let $v_j:=x_{U_j}$; by~\ref{enum:edgesfromU} we have $v_0\in X'_i$.

Recursively for $j=0,1,\dotsc,n$, we construct $M_i$-alternating paths $P_j$ of even length in $G'_i$ from $v_0$ to $v_j$ so that $P_j$ meets $U_j$ only in $v_j$, starting with the trivial path $P_0=v_0$. For $1\le j\le n$, since $P_{j-1}$ is an $M_i$-alternating path of even length in $G'_i$, its last edge (if existent) is in $M_i$. Hence by~\ref{enum:edgesfromU} every other edge in $\delta(U_{j-1})$, in particular $v^+_{j-1}w_{j-1}$, does not lie in $M_i$. As $v_{j-1}$ dominates $v_{j-1}^+$ via $U_{j-1}$, there is an $M_i$-alternating path $Q_{j-1}$ of even length in $G'_i[U_{j-1}]$ from $v_{j-1}$ to $v_{j-1}^+$. We can thus prolong $P_{j-1}$ to an $M_i$-alternating path $P_j$ from $v_0$ to a vertex in $U_j$: Let $P_j:=P_{j-1}v_{j-1}Q_{j-1}v_{j-1}^+w_{j-1}P^*u_jv_j^-$. We claim that $P_j$ has even length and that $v_j^-=v_j$. Indeed, as $u_j\in S'_i\setminus T'_i$, the $M_i$-alternating path $P_ju_j$ has odd length and thus $u_jv_j^-\in M_i$. As the only edge in $\delta(U_j)\cap M_i$ is incident with $v_j$, we have $v_j=v^-_j$ as desired.

If $U\in\cu_i$, we have thus constructed an $M_i$-alternating path $P=P_n$ in $G'_i$ whose last edge coincides with the last edge of $P^*$ and hence either both $P$ and $P^*$ have even length or they both have odd length. If $U\notin\cu_i$, then we can apply the same construction as before to obtain an $M_i$-alternating $v_0$--$U$~path $P$ from $P_n$ whose length has the same parity as the length of $P^*$. If this parity is even then the last vertex of $P$ is in $T'_i$ and hence in a set in $\cu_i$, which implies $T_i\subset\cu_i$. If the parity is odd then $U\notin\cu_i$ (as otherwise $P=P_n$ and this path has even length), hence $U$ is a vertex of $G'_i$ and lies in $S'_i\setminus T'_i$, which proves $S_i\subset S'_i\setminus T'_i$. This completes the proof of \ProR{pro:newOO}.
\end{proof}

\begin{proposition} \label{pro:ahalf}
  The function $\pi_{i+1}$ satisfies \eqref{eq:nonnegative} and \eqref{eq:undersaturated}.
\end{proposition}
\begin{proof}
By the definition of $\pi_{i+1}$ we have $\pi_{i+1}(U)=\frac12$ for every $U\in\cu_i$, thus every $U$ with $|\uuu U|>1$ begins its life with a positive potential. Since we only change potentials by $\frac12$, the potential of $U$ cannot obtain a negative value without becoming $0$ at some step $k$. But then $U$ is removed from $\OO_{k+1}$, so \eqref{eq:nonnegative} holds.

To prove that~\eqref{eq:undersaturated} holds, let $e=uv$ be an edge of $G$ and suppose that~\eqref{eq:undersaturated} does not hold for $e$ and $\pi_{i+1}$. Since it holds for $e$ and $\pi_i$ and we raised the potential only for sets in $T_i$, there is a set $U_1\in T_i$ (and hence $U_1\in\OO_{i+1}^{\MAX}$) with $e\in \delta(U_1)$, say $u\in \uuu U_1$ and $v\notin \uuu U_1$. Therefore, there is no set $U\in\OO_i$ with $\{u,v\}\subset\uuu U$. Since $\cv_i$ is laminar there is a unique set $U_2\in\cv_i^{\MAX}\setminus\{U_1\}$ with $e\in\delta(U_2)$, i.e.\ $v\in \uuu U_2$ and $u\notin \uuu U_2$. Clearly, we have
\begin{equation}\label{eq:raisedpotential}
  \sum_{\substack{U\in\OO_{i+1}\\ e\in\delta(U)}}\pi_{i+1}(U) - \sum_{\substack{U\in\OO_i\\ e\in\delta(U)}}\pi_i(U) =
  \begin{cases}
        0 & \text{if }U_2 \in S_i,\\
        1 & \text{if }U_2 \in T_i,\\
        \frac12 & \text{otherwise.}
  \end{cases}
\end{equation}
As~\eqref{eq:undersaturated} holds for $e$ and $\pi_i$ but not for $e$ and $\pi_{i+1}$, this means that $U_2\notin S_i$ (in particular $U_2\in\OO_{i+1}$).

Suppose that $\sum_{U\in\OO_i, e\in\delta(U)}\pi_i(U) = w(e)$, i.e.\ $e$ is present in $G_i$. Therefore, $U_1$ and $U_2$ are neighbours in $G_i|\cv_i^{\MAX}$ and~\eqref{eq:neighboursofT} yields $U_2\in S_i$, a contradiction. This means that $\sum_{U\in\OO_i, e\in\delta(U)}\pi_i(U) < w(e) < \sum_{U\in\OO_{i+1}, e\in\delta(U)}\pi_{i+1}(U)$. Thus $\sum_{U\in\OO_i, e\in\delta(U)}\pi_i(U) = w(e)-\frac12$ and $\sum_{U\in\OO_{i+1}, e\in\delta(U)}\pi_{i+1}(U) = w(e)+\frac12$ and hence $U_2\in T_i$ by~\eqref{eq:raisedpotential}.

For every vertex $x\in G$, define the \defi{$i$th energy of $x$} as $p_i(x):= \sum_{x\in \uuu U}\pi_i(U)$. As there is no $U\in\OO_i$ with $\{u,v\}\subset \uuu U$, we have $\sum_{U\in\OO_i, e\in\delta(U)}\pi_i(U) = p_i(u)+p_i(v)$ and hence $p_i(u)+p_i(v)=w(e)-\frac12$ is not an integer. We will see that this leads to a contradiction.

We claim that for every component $C$ of $G_i$ and any two vertices $x,y\in C$, the value $p_i(x)+p_i(y)$ is an integer (or equivalently: for every component $C$ of $G_i$ either the $i$th energy is an integer for all vertices in $C$ or it is not an integer for all vertices in $C$); indeed, if $xy$ is an edge of $G_i$ (it clearly suffices to consider this case) then it satisfies \eqref{eq:sat}. But then
\begin{equation*}
  w(xy) = \sum_{\substack{U\in \OO_i\\ xy\in \delta(U)}} \pi_i(U)= p_i(x) + p_i(y) - \sum_{\substack{U\in \OO_i\\ \{x,y\} \subset \uuu U}} 2\pi_i(U),
\end{equation*}
and as $w(xy)$ and $2\pi_i(U)$ for each $U$ are integers, our claim follows. As $G_i[\uuu U]$ is connected for every $U\in\OO_i$ (which follows immediately from the construction), the $i$th energy is either integral for every vertex in $U$ or non-integral for every vertex in $U$.

Furthermore, by applying \eqref{unmatched} recursively it is easy to show that for any set $X\in X_i$ there is precisely one vertex $x\in \uuu X$ such that the sets $U_x^j\in \OO_j^{MAX}$ with $x\in \uuu U_x^j$ have been unmatched by $M_j$ in every step $j$ of the construction and thus
\begin{equation} \label{eq:p}
  p_i(x)=\frac12i.
\end{equation}
By the definition of $T_i$, every element $U$ of $T_i$ lies in the same component of $G'_i$ as some $X\in X'_i$ and hence every vertex in $\uuu U$ lies in the same component of $G_i$ as any vertex in $\uuu X$. This easily implies that the $i$th energy is either integral for all vertices in $\bigcup_{U\in T_i}\uuu U$ (if $i$ is even) or non-integral for all such vertices (if $i$ is odd). As $u\in \uuu U_1\in T_i$ and $v\in \uuu U_2\in T_i$, this implies that $p_i(u)$ and $p_i(v)$ are either both integral or both non-integral, in particular, $p_i(u)+p_i(v)$ is integral, which yields the desired contradiction.
\end{proof}

\begin{proposition} \label{pro:term}
  The procedure terminates.
\end{proposition}
\begin{proof}
  We claim that after $i=\max_{e\in E(G)} w(e)$ steps (if not earlier) there is at most one unmatched vertex in $G'_i$. Suppose for contradiction that there are two, $U, Y$ say. There
  are vertices $u\in \uuu U$ and $y\in \uuu Y$ with $p_i(u)=p_i(y)=\frac12 i$, i.e.\ that satisfy~\eqref{eq:p}. Now the edge $uy$ lies in $G'_i$ since by \eqref{eq:p} $p_i(u)+p_i(y)= \max_{e\in E(G)} w(e) \geq w(uy)$, and this contradicts the maximality of $M_i$.
\end{proof}

Thus, after finitely many steps, $n$ say, we have a perfect or almost perfect matching $M_n$ in $G'_n$. By recursively applying condition~\eqref{cond:uam} we can extend $M_n$ to a perfect or almost perfect matching $M$ of $G$ with the additional property that

\begin{txteq}\label{MandOO}
  For every $U\in\OO_n$ we have $|M\cap\delta(U)|\in\{0,1\}$, and $|M\cap\delta(U)| = 0$ if and only if $M$ is almost perfect and $\uuu U$ contains the vertex unmatched by $M$.
\end{txteq}

We now claim that $M$ is strongly $w$-minimal.

Firstly, consider the case when $M$ is perfect. Pick any perfect matching $M'$ so that $M \triangle M'$ is finite, that is, there are disjoint finite edge-sets $N\subset M$ and $F \subset M'$ so that $M'=M - N + F$. By the definition of $G_i$ we have

\begin{equation} \label{eq:f}
  \sum_{e\in N} w(e) = \sum_{e\in N} \sum_{\substack{U\in \OO_n\\ e\in \delta(U)}}\pi_n(U),
\end{equation}

and by \eqref{eq:undersaturated} we have
\begin{equation} \label{eq:s}
  \sum_{e\in F} w(e) \geq \sum_{e\in F} \sum_{\substack{U\in \OO_n\\ e\in \delta(U)}}\pi_n(U).
\end{equation}

By~(\ref{MandOO}), for any element $U$ of $\OO_n$ there is at most one edge of $M$ in $\delta(U)$, thus $U$ appears in the first sum at most once. Moreover, as both $M$ and $M'$ are perfect, $F \cup N$ is a finite set of disjoint cycles and thus if $\pi_n(U)$ appears in the sum of~\eqref{eq:f} then it also appears in the sum of~\eqref{eq:s}. By the same argument, any $U$ with negative potential (hence $|U|=1$ by~\eqref{eq:nonnegative}) appearing in~\eqref{eq:s} also appears in~\eqref{eq:f}. Thus

\begin{equation} \label{eq:long}
\sum_{e\in N} \sum_{\substack{U\in \OO_n\\ e\in \delta(U)}}\pi_n(U) \leq \sum_{e\in F} \sum_{\substack{U\in \OO_n\\ e\in \delta(U)}}\pi_n(U),
\end{equation}

which by~\eqref{eq:f} and~\eqref{eq:s} implies that $\sum_{e\in N} w(e) \leq \sum_{e\in F} w(e)$. As $M'$ was chosen arbitrarily, this proves that $M$ is strongly $w$-minimal.

Next, consider the case when $M$ is almost perfect. There is only a difference to the previous case when $F$ meets the only vertex $x$ not matched by $M$, however \eqref{eq:long} remains true since by \eqref{eq:p} $x$ has maximum energy (in particular non-negative). Thus $M$ is strongly $w$-minimal also in this case.

\end{proof}

\begin{proof}[Proof of \Tr{main}]
Clearly, we may assume that all weights $w(e)$ are positive. Let $G'$ be the complete graph resulting from $G$ by adding an edge of weight $0$ between any two non-adjacent vertices of $G$, and define $w'(e):=- w(e)$ for every $e \in E(G')$. By \Tr{sminpm}, $G'$ has a strongly $w'$-minimal perfect or almost perfect matching $M$, and then $M':=M \cap E(G)$ is a strongly $w$-maximal matching of $G$. Indeed, suppose that there is a matching $M''$ where $M'' \triangle M'$ is finite such that
\begin{equation} \label{eq:w}
w[M''\backslash M']<w[M' \backslash M''].
\end{equation}
Let $L$ be the set of edges of $M \backslash M'$ that are incident with an edge of $M''\backslash M'$. Then, $N:=(M \cup (M''\backslash M')) \backslash (L \cup M'\backslash M'')$ is a matching in $G'$ with $N \triangle M$ finite, and since $w[L]=0$ we obtain $w[N\backslash M]<w[M \backslash N]$ by~\eqref{eq:w}. If $N$ leaves more than one vertex of $G'$ unmatched then, as $G'$ is complete, we can arbitrarily match all but at most one of those unmatched vertices to extend $N$ to a perfect or almost perfect matching of $G'$. As $w(e)\leq 0$ for every $e\in e(G')$, this contradicts the fact that $M$ is strongly $w$-minimal.
\end{proof}

\section{The non-rational case}

We now show that \Tr{sminpm} and \Tr{main} fail when we allow non-rational weights. Since \Tr{main} follows from \Tr{sminpm}, it suffices to construct a counterexample to the former. This counterexample $G$ will consist of two vertices $x$ and $y$, joined by infinitely many paths $P_1,P_2,\dotsc$. The idea is to choose the weights $w(e)$ so that a potential strongly $w$-maximal matching has to match both $x$ and $y$, and it has to match them in the same path $P_i$, and so that any such matching can be locally improved by changing it along $P_i \cup P_{i+1}$ so as to match $x$ and $y$ in $P_{i+1}$.

In order to achieve this situation, we will need an irrational value $a$ as a weight with the property that for every $\varepsilon>0$, there is an $n\in\N$ such that $na$ differs from some integer by less than $\varepsilon$. This is satisfied for instance for $a:=\sum_{i=1}^{\infty}10^{1-\frac12 i(i+1)} = 1.010010001\dotso$. The only weights in our graph will be $a$, $2a$, and $2a-1$. We will choose the paths $P_i$ so that each of them contains an odd number of edges, $2n_i+1$ say. Every second edge on $P_i$ will have weight $2a-1$, while the remaining $n_i+1$ edges on $P_i$ will have weights $a$ and $2a$, and the sum of their weights will be larger than $n_i(2a-1)$, i.e.\ than the sum of the weights of the other edges, by a value that is strictly increasing with $i$.

First, let us define the numbers $n_i$. Let $n_1:=1$ and, for $i=1,2,\dotsc$, let $n_{i+1}:=10^{i+1}n_i+1$. (Thus, $n_2=101,n_3=101001$ etc.) It is not hard to check that
\begin{equation}\label{eq:almostintegral}
  10^{-(i+1)}<10^{\frac12 i(i+1)-1}a-n_i<10^{-i}.
\end{equation}
We write $P_i=x^i_0x^i_1\ldots x^i_{2n_i}x^i_{2n_i+1}$, where $x=x^i_0$ and $y=x^i_{2n_i+1}$. As already mentioned, we put $w(e):=2a-1$ for each edge $e=x^i_{2j-1}x^i_{2j},~1\le j\le n_i$. We call these edges the \emph{even edges} of $P_i$; the other edges on $P_i$ are the \emph{odd edges} of $P_i$. Define the weights of the odd edges of $P_i$ as follows. Inductively, for $k=0,1,\ldots,n_i$, we put
\begin{equation}\label{eq:evenweight}
  w(x^i_{2k}x^i_{2k+1}):=
  \begin{cases}
    2a & \text{if }\sum_{j=0}^{k-1}w(x^i_{2j}x^i_{2j+1}) < k(2a-1)\\
    a & \text{otherwise}
  \end{cases}
\end{equation}
By this definition, we achieve that on every subpath $xP_ix^i_{2k}$ of $P_i$, the sums of weights of the even edges (which equals $k(2a-1)$) and of the odd edges do not differ too much. Indeed, it is easy to check that
\begin{equation}\label{eq:evenodd}
  1-a\le\sum_{j=0}^{k-1}w(x^i_{2j}x^i_{2j+1})-k(2a-1)<1.
\end{equation}
Given a subpath $P$ of some $P_i$, we write $even(P)$ (respectively $odd(P)$) for the sum of the weights of the even (resp.\ odd) edges of $P_i$ on $P$. With this notation and~\eqref{eq:evenodd}, we have the two inequations
\begin{align}
  \label{eq:evenlength}
  odd(xP_ix^i_k)-even(xP_ix^i_k)&<1 \text{ for $k$ even,\qquad and}\\
  \label{eq:oddlength}
  odd(xP_ix^i_k)-even(xP_ix^i_k)&\ge a \text{ for $k$ odd.}
\end{align}

Suppose there is is a strongly $w$-maximal matching $M$ in $G$. First, we show that on each $P_i$ there is at most one unmatched vertex. Indeed, if there are at least two unmatched vertices on some $P_i$, then we can pick two of them $x^i_j$ and $x^i_k$ with $j<k$ so that all vertices $x^i_l$ with $j<l<k$ are matched. Note that the path $P=x^i_jP_ix^i_k$ has odd length. If $j$ is even then $k$ is odd, and we have $odd(P)-even(P)=odd(xP_ix^i_k)-even(xP_ix^i_k)-\big(odd(xP_ix^i_j)-even(xP_ix^i_j)\big)>a-1>0$. If $j$ is odd, we have by a similar calculation again $even(P)-odd(P)>a-1>0$. This means that we can replace the edges in $M\cap E(P)$ by the edges in $E(P)\setminus M$ and improve $M$, a contradiction. Therefore, every $P_i$ contains at most one unmatched vertex. In particular, $x$ and $y$ cannot both be unmatched.

Thus one of $x,y$, say $x$, is matched in $M$, to $x^i_1$ say. If $y=x^i_{2n_i+1}$ is unmatched and $P_i$ has odd length, there has to be another unmatched vertex on $P_i$, which again leads to a contradiction. Thus, $y$ is matched in $M$, to $x^j_{2n_j}$ say. Easily, for $k\not= i,j$ each vertex on $P_k$ is matched. Suppose $i\not= j$; then there are unmatched vertices $x^i_m$ and $x^j_n$. Since no other vertex on $P_i\cup P_j$ is unmatched, $m$ is even and $n$ is odd. Furthermore, the path $P:=x^i_mP_ixP_jx^j_n$ is an $M$-alternating path; we claim that replacing the edges in $M\cap E(P)$ by those in $E(P)\setminus M$ is an improvement of $M$. Indeed, on $x^i_mP_ix$, we replace the odd edges by the even ones and lose less than $1$ by \eqref{eq:evenlength}, while on $xP_jx^j_n$, we replace the even edges by the odd ones and gain at least $a$ by \eqref{eq:oddlength}. Since $a>1$, this contradicts the strong $w$-maximality of $M$ and hence $i=j$.

Thus, $M$ is a perfect matching. We claim that we can improve $M$ by replacing its edges in $P_i\cup P_{i+1}$ by those in $E(P_i\cup P_{i+1})\setminus M$. Indeed, $M$ consists of the odd edges of $P_i$ and the even edges of all the other $P_j$. Clearly, we have $even(P_j)=even(xP_jx^j_{2n_j})=n_j(2a-1)$ and $odd(P_j)=odd(xP_jx^j_{2n_j})+w(x^j_{2n_j}x^j_{2n_j+1})$ for every $j$, and if $k_j$ denotes of odd edges of $xP_jx^j_{2n_j}$ with weight $a$, then we have $odd(P_j)=n_j2a-k_ja+w(x^j_{2n_j}x^j_{2n_j+1})$ and hence
\begin{equation*}
  odd(P_j)-even(P_j) = n_j - k_ja + w(x^j_{2n_j}x^j_{2n_j+1}).
\end{equation*}
If $k_j<10^{\frac12 j(j+1)-1}$ then $odd(xP_jx^j_{2n_j})-even(xP_jx^j_{2n_j}) = n_j - k_ja > a-10^{-(j+1)}$ by~\eqref{eq:almostintegral}, which contradicts~\eqref{eq:evenlength} as $a-10^{-(j+1)}>1$. On the other hand, if $k_j>10^{\frac12 j(j+1)-1}$ then $odd(xP_jx^j_{2n_j})-even(xP_jx^j_{2n_j}) = n_j - k_ja < -a-10^{-j}$ by~\eqref{eq:almostintegral}, which contradicts~\eqref{eq:evenodd}. Thus, $k_j=10^{\frac12 j(j+1)-1}$ and $-10^{-j} < odd(xP_jx^j_{2n_j})-even(xP_jx^j_{2n_j}) < -10^{-(j+1)} < 0$. By~\eqref{eq:evenweight} we have $w(x^j_{2n_j}x^j_{2n_j+1})=2a$ and thus
\begin{equation*}
  2a-10^{-j} < odd(P_j)-even(P_j) < 2a-10^{-(j+1)}.
\end{equation*}
In particular, $odd(P_i)-even(P_i) < odd(P_{i+1})-even(P_{i+1})$ and hence we can improve $M$ by using the even edges of $P_i$ and the odd edges of $P_{i+1}$ instead of the odd edges of $P_i$ and the even edges of $P_{i+1}$. Thus we get a contradiction, proving that $G$ has no strongly $w$-maximal matching.

\bibliographystyle{plain}
\bibliography{graphs}

\begin{thebibliography}{1}

\bibitem{infinitetutte}
R.~Aharoni.
\newblock Matchings in infinite graphs.
\newblock {\em J. Combin. Th., Ser. B}, 44:87--125, 1988.

\bibitem{infinitemenger}
R.~Aharoni and E.~Berger.
\newblock {M}enger's theorem for infinite graphs.
\newblock Preprint.

\bibitem{aharoniziv}
R.~Aharoni and R.~Ziv.
\newblock Lp duality in infinite hypergraphs.
\newblock {\em J. Combin. Th., Ser. B}, 50:82--92, 1988.

\bibitem{diestelBook05}
R.~Diestel.
\newblock {\em Graph Theory \emph{(3rd edition)}}.
\newblock Springer-Verlag, 2005.
\newblock \\ Electronic edition available at:\\ {\small\tt
  http://www.math.uni-hamburg.de/home/diestel/books/graph.theory}.

\bibitem{edmonds}
J.~Edmonds.
\newblock Maximum matching and a polyhedron with 0,1-vertices.
\newblock {\em Journal of Research National Bureau of Standards Section B},
  69:125--130, 1965.

\bibitem{schrijverBook}
A.~Schrijver.
\newblock {\em Combinatorial Optimization - Polyhedra and Efficiency}.
\newblock Springer-Verlag, 2003.

\bibitem{steffens}
K.~Steffens.
\newblock Matchings in countable graphs.
\newblock {\em Canad. J. Math}, 29:165--168, 1976.

\end{thebibliography}

\end{document}